\documentclass[10pt]{article}

\usepackage{amsmath,amsfonts,amssymb}
\usepackage{epsf}
\RequirePackage{amsfonts,amssymb,amsmath,amscd,amsthm}
\usepackage{epsfig,tabularx,graphicx,float}
\usepackage{graphicx}
\usepackage[a4paper]{geometry}

\usepackage{parskip}
\usepackage[utf8]{inputenc}
\usepackage{times}
\usepackage{float}
\usepackage{multicol}
\usepackage{multirow}

\usepackage{float}
\newtheorem{theorem}{Theorem}[section]
\newtheorem{proposition}[theorem]{Proposition}

\newtheorem{definition}[theorem]{Definition}
\newtheorem{corollary}[theorem]{Corollary}

\newtheorem{lemma}[theorem]{Lemma}
\begin{document}

\begin{flushleft}
Prasadini Mahapatra \footnote{\textbf{Prasadini Mahapatra}, Department of Mathematics, National Institute of Technology Rourkela, Odisha, India} and Divya Singh \footnote{\textbf{Corresponding Author: Divya Singh}, Department of Mathematics, National Institute of Technology Rourkela, Odisha, India}
\end{flushleft}

\begin{flushleft}
\textbf{\Large Shift invariant space and FMRA on Vilenkin group}\\[1 cm]
\end{flushleft}

\noindent \textbf{Abstract:} 
We construct the spectrum for a shift invariant space on Vilenkin group. We prove the results related to spectrum and Frame multiresolution analysis for Cantor dyadic group and Vilenkin group.



\section{Introduction}\label{sec1}
In the theory of wavelet, Multiresolution analysis (MRA) is used in broad areas to obtain wavelets, scaling functions and frames.   In recent years, the more general theory is developed such as Generalized multiresolution analysis(GMRA), Frame multiresolution analysis(FMRA), Parseval frame multiresolution analysis(PFMRA) etc. The authors like Schipp, Wade, Simon, Golubov, Efimov, Skvortsov, Maqusi, Siddiqi, Beauchamp developed the theory related to Walsh analysis. In 1923, J. Walsh introduced Walsh functions as a linear combination of Haar functions. In the 1940s Gelfand recognized that Walsh functions are identified with characters of the Cantor dyadic group. There is a lot of work done on locally compact abelian group. For every locally compact abelian group, there exists a Haar measure, which is not identically zero but unique up to a multiplicative constant. N. J. Fine and N. Ya Vilenkin independently determined that the Walsh system is the group of characters of the Cantor dyadic group. A large class of locally compact abelian groups was introduced by  N. Ya Vilenkin, called Vilenkin groups. Cantor dyadic group is its particular case.
Lang \cite{lang1} determined the compactly supported orthogonal wavelets on the locally compact Cantor dyadic group. Shift operators such as dilation operator, multiresolution analysis(MRA) were built and some necessary regularity conditions for wavelets and sufficient conditions on scaling filters were given to occur orthonormal wavelets. The generalised Walsh functions form an orthonormal system in the Vilenkin group $G$ and mask of refinable equation is given in terms of these generalised Walsh functions. Refinable equation gives refinable function which generates MRA and hence wavelets, if the mask satisfies certain conditions. Necessary and sufficient conditions were given over the mask of scaling function $\phi$ in terms of modified Cohen's condition and blocked sets such that $\phi$ generates an MRA. Lang \cite{lang2} examined that the necessary and sufficient conditions were given on a trigonometric polynomial scaling filter resulting in a multiresolution analysis.

In the last decade, several authors such as Aldroubi, Benedetto, Bownik, de Boor, DeVore, Li, Ron, Rzeszotnik, Shen, Weiss and Wilson, Behera have been studied the shift invariant subspaces for $L^{2}$ space of $\mathbb{R}^{d}$ by different aspects. This theory plays an important role in many areas, especially in the theory of wavelets, and multiresolution
analysis. Shift-invariant spaces are very important in applications and the theory had a great development in the last twenty years, mainly in approximation theory, sampling, wavelets, and frames. In particular, they serve as models in many problems in signal and image processing. The concept of frames was first introduced in 1952 by R. Duffin and A. C. Schaefer in their work on nonharmonic analysis. 

This paper is organized as follows. It consists of two sections. In Section 2, we study the basic preliminaries on Vilenkin group and some results which is used later in the main results. Section 3 contains the spectrum of shift invariant spaces and frame multiresolution analysis for Cantor dyadic group and Vilenkin group.

\section{Preliminaries}\label{sec2}

The Vilenkin group $\tilde{G}$ is defined as the group of sequences
$$
x=(x_{j})=(..., 0, 0, x_{k}, x_{k+1}, x_{k+2},...),
$$
where $x_{j}\in \lbrace 0, 1, ..., p-1\rbrace $ for $ j \in \mathbb{Z}$ and $ x_{j}=0 $, for $ j < k = k(x)$. The group operation on $\tilde{G}$, denoted by $ \oplus $, is defined as coordinatewise addition modulo $ p $:
$$
(z_{j})=(x_{j})\oplus (y_{j})\Leftrightarrow z_{j}=x_{j} + y_{j}(\text{mod}\: p), \text{ for } j \in \mathbb{Z}.
$$
Let
$$
U_{l}=\lbrace(x_{j}) \in \tilde{G} : x_{j}=0\: \text{for}\: j\leq l\rbrace ,\qquad l\in \mathbb{Z},
$$
be a system of neighbourhoods of zero in $\tilde{G}$. In case of topological groups if we know neighbourhood system $\lbrace U_l \rbrace_{l \in \mathbb Z}$ of zero, then we can determine neighbourhood system of every point $x=(x_j) \in G$ given by $\lbrace U_l \oplus x \rbrace_{l \in \mathbb Z}$, which in turn generates a topology on $\tilde{G}$.

Let $ U=U_{0} $ and $ \ominus $ denotes the inverse operation of $ \oplus $. The Lebesgue spaces $ L^{q}(\tilde{G}),\: 1\leq q\leq \infty $, are defined by the Haar measure $ \mu $ on Borel's subsets of $\tilde{G}$ normalized by $ \mu(U)=1 $.

The group dual to $\tilde{G}$ is denoted by $ \tilde{G}^{*} $ and consists of all sequences of the form
$$
\omega=(\omega_{j})=(...,0,0,\omega_{k},\omega_{k+1},\omega_{k+2},...),
$$
where $ \omega_{j} \in \lbrace 0,1,...,p-1\rbrace $, for $ j \in \mathbb{Z} $ and $ \omega_j=0 $, for $ j<k=k(\omega) $. The operations of addition and subtraction, the neighbourhoods $\lbrace U_{l}^{*}\rbrace $ and the Haar measure $ \mu^{*} $ for $\tilde{G}^{*} $ are defined as above for $\tilde{G}$. Each character on $\tilde{G}$ is defined as
$$
\chi(x,\omega)=\exp\bigg(\frac{2\pi i}{p}\sum_{j\in \mathbb{Z}}{x_{j}w_{1-j}}\bigg),\quad x \in \tilde{G},
$$
for some $ \omega \in \tilde{G}^{*} $.

Let $ H = \lbrace (x_{j}) \in \tilde{G}\: |\: x_{j} =0 \; \text{ for } j>0\rbrace $ be a discrete subgroup in $\tilde{G}$  and $A$ be an automorphism on $\tilde{G}$ defined by $ (Ax)_j=x_{j+1} $, for $x=(x_j) \in \tilde{G}$. From the definition of annihilator and above definition of character $\chi$, it follows that the annihilator $ H^{\perp} $ of the subgroup $ H $ consists of all sequences $ (\omega_j) \in \tilde{G}^{*} $ which satisfy $ \omega_j=0 $ for $ j>0 $.

Let $ \lambda:\tilde{G}\longrightarrow \mathbb{R}_+ $ be defined by
\begin{equation*}
\lambda(x)=\sum_{j \in \mathbb{Z}}{x_{j} p^{-j}}, \qquad x=(x_j) \in \tilde{G}.
\end{equation*}
It is obvious that the image of $ H $ under $ \lambda $ is the set of non-negative integers $ \mathbb{Z}_{+} $. For every $ \alpha \in \mathbb{Z}_{+} $, let $ h_{[\alpha]} $ denote the element of $ H $ such that $ \lambda(h_{[\alpha]})=\alpha $. For $\tilde{G}^{*} $, the map $ \lambda^{*}:\tilde{G}^{*}\longrightarrow \mathbb{R}_{+} $, the automorphism $ B \in \text{Aut }\tilde{G}^{*} $, the subgroup $ U^{*} $ and the elements $ \omega_{[\alpha]} $ of $ H^{\perp} $ are defined similar to $ \lambda$, $A$, $U$ and $h_{[\alpha]}$, respectively. %Further, $ \chi(Ax,\omega)=\chi(x,B\omega) $ for all $ x \in G$, $\omega \in G^{*} $.

The generalised Walsh functions for $\tilde{G}$ are defined by
\begin{equation*}
W_\alpha(x)=\chi(x,\omega_{[\alpha]}),\qquad \alpha \in \mathbb{Z}_+, x \in \tilde{G}.
\end{equation*}
These functions form an orthogonal set for $L^2(U)$, that is,
\begin{equation*}
\int_{U}{W_{\alpha}(x) \overline{W_{\beta}(x)}d\mu(x)}=\delta_{\alpha,\beta},\qquad \alpha,\beta \in \mathbb{Z}_{+},
\end{equation*}
where $ \delta_{\alpha,\beta} $ is the Kronecker delta. The system $ {W_{\alpha}} $ is complete in $ L^{2}(U) $. The corresponding system for $ \tilde{G}^{*} $ is defined by
\begin{equation*}
W_{\alpha}^{*}(\omega)=\chi(h_{[\alpha]},\omega),\qquad \alpha \in \mathbb{Z}_{+}, \omega \in \tilde{G}^{*}.
\end{equation*}
The system $\lbrace W_{\alpha}^{*} \rbrace$ is an orthonormal basis of $ L^{2}(U^{*}) $.

For positive integers $ n $ and $ \alpha $,
\begin{equation*}
U_{n,\alpha}=A^{-n}(h_{[\alpha]}) \oplus A^{-n}(U).
\end{equation*}

\begin{definition}
A sequence $\{f_{j}:j\in J\}$ in a Hilbert space $H$ is called a Bessel sequence if there exists $B>0$ such that 
$$
\sum_{j \in J} \vert \langle f,f_{j} \rangle \vert ^{2}\leq B \Vert f \Vert^{2}, \quad \forall f \in H.
$$
\end{definition}

\begin{definition}
A sequence $\{f_{j}:j\in J\}$ in a Hilbert space $H$ is called a \emph{frame} if there exist constants $a$ and $b$, $0 <a \leq b < \infty$, such that
$$
a \Vert f \Vert ^{2} \leq \sum_{j \in J} \vert \langle f,f_{j} \rangle \vert ^{2}\leq b \Vert f \Vert^{2}, \quad \forall f \in H.
$$ 
\end{definition}
If $a=b$, then $(f_j)$ is called a \emph{tight frame}, and it is called a \emph{Parseval frame} if $a=b=1$.

The principal shift invariant space generated by $\psi$ is denoted as $\langle \psi \rangle$, defined by $\langle \psi \rangle=\overline{\text{span}}\{\psi(. \ominus h):h \in H\}$.  
In general, a closed subspace $V \subseteq L^{2}(\tilde{G})$ is called shift invariant if and only if $T_{h}V \subset V$, for all $h \in H$. 
For $\phi \in L^2(\tilde{G})$, the periodization function is defined by 
\begin{eqnarray} 
P_{\phi}(\omega)=\sum_{h \in H^{\bot}} \vert \hat{\phi}(\omega \oplus h) \vert ^{2}.
\end{eqnarray}

For $f,g \in L^{2}(\tilde{G})$, the \emph{$C$-bracket} is defined as
\begin{equation}\label{110}
[f,g](x):=\sum_{h \in H} f(x \oplus h)\overline{g(x \oplus h)}.
\end{equation}

Clearly, $P_{\phi}=[\hat{\phi},\hat{\phi}]$.

\begin{proposition}\label{psi_char}\cite{mahapatra}
Let $(T_{h}\phi)_{h \in H}$ be a Bessel sequence. Then for each $H^{\bot}$-periodic function $m \in L^{2}(U^{*})$ we have $m\hat{\phi}\in L^{2}(G^{*})$. Here $m$ is extended by $H^{\bot}$-periodicity to the function $m$ on $L^{2}(G^{*})$. Moreover, if $(T_{h}\phi)_{h \in H}$ is a frame for $\langle \phi \rangle$, then
\begin{equation} 
\langle \phi \rangle=\{f \in L^{2}(G):\hat{f}=m\hat{\phi}, \; m \in L^{2}(U^{*})\}
\end{equation}	
\end{proposition}

\begin{theorem}\label{onb}\cite{mahapatra}
For $\phi \in L^{2}(G)$, the system $(T_{h}\phi)_{h \in H}$ is a Parseval frame for $\langle\phi \rangle$ iff there exists a $H^{\bot}$-periodic set $\Omega \subseteq G^{*}$ such that $P_{\phi}=\chi_{\Omega}$ a.e. In particular, $(T_{h}\phi)_{h \in H}$ is an orthonormal basis for $\langle \phi \rangle$ iff $P_{\phi}(\omega)=1$, for a.e. $\omega \in G^{*}$.
\end{theorem}
	
\begin{theorem}\label{frame seq char1}\cite{mahapatra}
The sequence $\{T_{h}\phi:h \in H\}$ in $L^{2}(G)$ is a frame sequence with frame bounds $C$ and $D$ iff 
\begin{equation}\label{115}
C \leq [\hat{\phi},\hat{\phi}](\omega) \leq D \text{ for a.e. $\omega$ in }U^{*}\setminus \mathcal{N}_{\phi},
\end{equation}
where $\mathcal{N}_{\phi}=\{\omega \in U^{*}:[\hat{\phi},\hat{\phi}](\omega)=0\}$.
\end{theorem}

\section{Spectrum of shift invariant spaces and FMRA}\label{6-2}

For $f \in L^2(\tilde{G})$, let $\mathbb E_{f}=\text{ supp }P_f=\{\omega \in \tilde{G}^{*}:P_{f}(\omega) \neq 0\}$. Let $\hat{f}_{\Vert\omega}=(\hat{f}(\omega \oplus h))_{h \in H^{\bot}}$. Since $P_f \in L^1(U^*)$, therefore $\hat{f}_{\Vert\omega}$ belongs to $l^2(\mathbb N_0)$, for a.e. $\omega \in U^{*}$. Let $\mathbb S \subset L^{2}(\tilde{G})$, then $\hat{\mathbb S}_{\Vert\omega}=\{\hat{f}_{\Vert\omega}:f \in \mathbb S\}$. Let $\mathbb S$ be a shift-invariant subspace of $L^{2}(\tilde{G})$, then the spectrum $\eta(\mathbb S)$ of $\mathbb S$ is defined by $\eta(\mathbb S)=\{\omega \in U^{*}:\hat{\mathbb S}_{\Vert\omega}\neq \{0\}\}$. For $f \in L^2(\tilde{G})$, let $\mathbb V_{f}=\{\omega \in U^{*}: \hat{f}_{\Vert\omega} \neq 0\}$. For $E \subseteq U^*$, let $E^{\sim}=E \oplus H^{\perp}$.

\begin{proposition}\label{606}
For $\phi \in L^2(\tilde{G})$, we have
$$
\eta(\langle \phi\rangle)=\mathbb V_{\phi}=\mathbb E_{\phi}\cap U^{*}.
$$
\end{proposition}

\begin{proof}
For $f \in \langle \phi \rangle$, we have $\hat{f}(\omega)=m(\omega)\hat{\phi}(\omega)$, for some $m \in M_{\phi}$. Then $\hat{f}_{\Vert\omega}=(m(\omega \oplus h)\hat{\phi}(\omega \oplus h))_{h \in H^{\bot}}$.

Let $\omega \in \eta(\langle \phi \rangle)$, then there exists $f \in \langle \phi \rangle$ such that
$
\hat{f}_{\Vert\omega} \neq 0.
$
Thus, $\hat{\phi}_{\Vert\omega} \neq 0$ implies that $\omega \in \mathbb V_{\phi}$. Therefore, $\eta(\langle \phi \rangle)\subseteq \mathbb V_{\phi}$.

Further, if $\omega \in  \mathbb V_{\phi}$, then $\hat{\phi}_{\Vert\omega}\neq 0$, and hence $\omega \in \eta(\langle \phi \rangle)$. Thus, $\mathbb V_{\phi} \subseteq \eta(\langle \phi \rangle)$.

For $\omega \in \mathbb V_{\phi}$, we have $\hat{\phi}_{\Vert\omega} \neq 0$ and hence $P_{\phi}(\omega)>0$. Thus, $\mathbb V_{\phi} \subseteq \mathbb E_{\phi} \cap U^*$. Similarly, $\mathbb E_{\phi} \cap U^{*} \subseteq  \mathbb V_{\phi}$. Thus, we have $\eta(\langle \phi \rangle)=\mathbb V_{\phi}=\mathbb E_{\phi} \cap U^{*}$.
\end{proof}

\begin{corollary}\label{607}
For the principal shift invariant space $\langle \phi \rangle$, the following are equivalent:
\begin{enumerate}
\item[(i)] $\{T_{h}\phi: h \in H\}$ is a Parseval frame for $\langle \phi \rangle$.
\item[(ii)] $\sum_{h \in H^{\bot}}|\hat{\phi}(\omega \oplus h)|^{2}=\mathbf{1}_{\mathbb E_{\phi}}(\omega)$ a.e.
\item[(iii)] $\sum_{h \in H^{\bot}}|\hat{\phi}(\omega \oplus h)|^{2}=\mathbf{1}_{\eta(\langle \phi \rangle)\oplus H^{\bot}}(\omega)$ a.e.
\end{enumerate}
\end{corollary}

\begin{corollary}\label{608}
Suppose that the principal shift invariant spaces generated by $\phi$ and $\psi$ are equal i.e $\langle \phi \rangle=\langle \psi \rangle$. Then, we have $\mathbb E_{\phi}=\mathbb E_{\psi}$.
\end{corollary}

The dilation operator on $L^2(\tilde{G})$ is given by,
\begin{equation*}
\mathbb Df(x)=p^{\frac{1}{2}}f(Ax), \quad f \in L^{2}(\tilde{G}).
\end{equation*}
Note that
\begin{equation*}
\widehat{\mathbb D^{j}f}(\omega)=p^{-j/2}\hat{f}(B^{-j}\omega), \text{ for every } j \in \mathbb{Z}
\end{equation*}

We have already given the definition of MRA in $L^2(\tilde{G})$. As given in \cite{behera}, we have following characterization of scaling function of an MRA in $L^2(\tilde{G})$ for a prime $p$.

\begin{theorem}\label{scaling function mra}
A function $\phi \in L^{2}(\tilde{G})$ is a scaling function for an MRA of $L^{2}(\tilde{G})$ if and only if
\begin{enumerate}
\item[(i)] $\sum_{h \in H^{\perp}}|\hat{\phi}(\omega \oplus h)|^2=1$, for a.e. $\omega \in \tilde{G}^{*}$,
\item[(ii)] $lim_{j \rightarrow \infty}|\hat{\phi}(B^{-j}\omega)|=1$, for a.e. $\omega \in \tilde{G}^{*}$,
\item[(iii)] there exists $H^{\bot}$-periodic function $m_{0}$ in $L^{2}\left(U^{*}\right)$ such that  $\hat{\phi}(B\omega)=m_{0}(\omega)\hat{\phi}(\omega)$, for a.e. $\omega \in \tilde{G}^{*}$.
\end{enumerate}
\end{theorem}

Frame multiresolution analysis (FMRA) for $L^2(\tilde{G})$ is defined as follows. The only difference between MRA and FMRA is in the last condition corresponding to scaling function.

\begin{definition}
A sequence $\lbrace V_j:j \in \mathbb Z \rbrace$ of closed subspaces of $L^2(\tilde{G})$ is said to be a frame multiresolution analysis, if
\begin{enumerate}
\item[(i)] $V_j \subseteq V_{j+1}$, for each $j \in \mathbb Z$;
\item[(ii)] $f \in V_j$ iff $f(A\cdot) \in V_{j+1}$, for each $j \in \mathbb Z$;
\item[(iii)] $\overline{\cup_{j \in \mathbb Z} V_j}=L^2(\tilde{G})$;
\item[(iv)] $\cap_{j \in \mathbb Z} V_j=\lbrace 0 \rbrace$;
\item[(v)] there exists a function $\phi \in V_0$, called scaling function of the FMRA, such that $\lbrace T_h\phi:h \in H \rbrace$ is  a frame for $V_0$.
\end{enumerate}
\end{definition}

If $\lbrace T_h\phi:h \in H \rbrace$ forms a Parseval frame for $V_0$, then $\lbrace V_j:j \in \mathbb Z \rbrace$ is called a Parseval FMRA.

Let $\phi \in L^2(\tilde{G})$ be a scaling function for the Parseval FMRA $\lbrace V_j:j \in \mathbb Z \rbrace$. For $j \in \mathbb Z$, suppose that $f \in V_j=\overline{\text{span}}\lbrace\mathbb D^jT_h \phi: h \in H\rbrace$.  From Proposition \ref{psi_char}, for each $j \in \mathbb Z$, we have
$$
V_j=\lbrace f \in L^2(\tilde{G}):\hat{f}(B^j\omega)=m(\omega)\hat{\phi}(\omega) \text{ a.e. for some } m \in L^2(U^*) \rbrace.
$$

For each $f \in V_j$, there exists the minimal filter $m_f=[\hat{f}(B^j\cdot), \hat{\phi}] \in L^2(U^*)$ such that $\hat{f}(B^j\omega)=m_f(\omega)\hat{\phi}(\omega)$ a.e.. %and $m(\omega)=0$, for every $\omega \in \mathbb E_{\phi}$.

Since $\phi \in V_0 \subseteq V_1$, we have $\hat{\phi}(B\omega)=m_{\phi}(\omega)\hat{\phi}(\omega)$ a.e. Then the function $m_{\phi}$ is called the \textit{minimal low-pass filter}.

\begin{lemma}
Let $\phi \in L^2(\tilde{G})$. Let $f_i \in \langle \phi \rangle$, $i=1,2$, and $m_i \in L^2(U^*)$ be such that $\hat{f}_i(B\omega)=m_i(\omega)\hat{\phi}(\omega)$ a.e. Then
$$
[\hat{f}_1, \hat{f}_2](B\omega)=\sum_{\zeta=0}^{p-1} m_1(\omega \oplus 0.\zeta)\overline{m_2(\omega \oplus 0.\zeta)}P_{\phi}(\omega \oplus 0.\zeta), \text{ a.e.},
$$
where $m_1$ and $m_2$ are extended $H^{\perp}$-periodically to $\tilde{G}^*$
\end{lemma}

\begin{proposition}\label{low pass fmra}
Let $\lbrace V_j: j \in \mathbb Z \rbrace$ be a Parseval FMRA with scaling function $\phi$ and minimal low-pass filter $m_{\phi}$. Then
\begin{equation}\label{low pass fmra filter}
\sum_{\zeta=0}^{p-1} |m_{\phi}(\omega \oplus 0.\zeta)|^2=\mathbf{1}_{\mathbb E_{\phi}}(B\omega), \text{ for a.e. } \omega.
\end{equation}
\end{proposition}

\begin{proof}
From the previous lemma, for $f_1=f_2=\phi$ and $m_1=m_2=m_{\phi}$, we have
$$
[\hat{\phi}, \hat{\phi}](B\omega)=\sum_{\zeta=0}^{p-1}\vert m_{\phi}(\omega \oplus 0.\zeta) \vert^2P_{\phi}(\omega \oplus 0.\zeta), \text{ a.e.}
$$
Then the result follows from the definition of minimal low pass filter, and the fact that $\lbrace T_h\phi:h \in H \rbrace$ is a Parseval frame for $V_0$.
\end{proof}

\begin{theorem}\label{scaling function fmra}
If $\phi \in L^{2}(\tilde{G})$ is a scaling function for a Parseval FMRA $\lbrace V_j:j \in \mathbb Z \rbrace$, then
\begin{enumerate}
\item[(i)] $\sum_{h \in H^{\perp}}|\hat{\phi}(\omega \oplus h)|^2=\mathbf{1}_{\mathbb E_{\phi}}(\omega)$, for a.e. $\omega \in \tilde{G}^{*}$,
\item[(ii)] there exists $H^{\bot}$-periodic function $m \in L^{2}\left(U^{*}\right)$ such that $\hat{\phi}(B\omega)=m(\omega)\hat{\phi}(\omega)$, for a.e. $\omega \in \tilde{G}^{*}$,
\item[(iii)] $lim_{j \rightarrow \infty}|\hat{\phi}(B^{-j}\omega)|=1$, for a.e $\omega \in \tilde{G}^{*}$.
\end{enumerate}
\end{theorem}

\begin{proof}
Statements (i) and (ii) follow trivially. To prove (iii), let $P_j$ be the orthogonal projection of $L^2(\tilde{G})$ onto $V_j$, for $j \in \mathbb Z$. Since $\overline{\cup_{j \in \mathbb Z}V_j}=L^2(\tilde{G})$, we have $\Vert P_jf \Vert \rightarrow \Vert f \Vert$, for every $f \in L^2(\tilde{G})$ as $j \rightarrow \infty$. Let $\hat{f}=\mathbf 1_{U^*}$. Since $\lbrace T_h\phi:h \in H \rbrace$ is a Parseval frame for $V_0$, therefore $\lbrace\mathbb D^j T_h \phi : h \in H\rbrace$ is a Parseval frame for $V_j$, for each $j \in \mathbb Z$. Thus, for large enough positive integer $j$
\begin{align*}
\Vert P_jf \Vert^2 & = \sum_{h \in H}\vert \langle \hat{f},\hat{\phi}_{j,h} \rangle \vert^2\\
                   & = p^j \sum_{h \in H} \vert \int_{U^*} \mathbf 1_{B^{-j}U^*}(\omega)\overline{\hat{\phi}(\omega)} \chi(h,\omega) \; d\omega \vert^2 \\
                   & = p^j \int_{U^*} \vert \mathbf{1}_{B^{-j}U^*}(\omega)\overline{\hat{\phi}(\omega)} \vert^2 \; d\omega \\              
                   & = \int_{U^*} \vert \hat{\phi}(B^{-j}\omega) \vert^2 \; d\omega.
\end{align*}
Since $\Vert P_jf \Vert^2 \rightarrow 1$, as $j \rightarrow \infty$, therefore
\begin{equation}\label{projection}
\lim_{j \rightarrow \infty}\int_{U^*} \vert \hat{\phi}(B^{-j}(\omega)) \vert^2 \; d\omega =1.
\end{equation}
Now, consider the minimal filter $m_{\phi}$ for $\phi$, then
\begin{align*}
\hat{\phi}(\omega)& = (\Pi_{j=1}^{N}m_{\phi}(B^{-j}\omega))\hat{\phi}(B^{-N}\omega), \text{ for every } N\in \mathbb N.
\end{align*}
From Equation \ref{low pass fmra filter}, $\vert m_{\phi}(\omega) \vert \leq 1$ a.e. and hence $(\vert \hat{\phi}(B^{-j}(\omega)) \vert)$ is an increasing sequence as $j \rightarrow \infty$, a.e. Let $F(\omega)=\lim_{j \rightarrow\infty} \vert \hat{\phi}(B^{-j}\omega) \vert^2$. Further, from Theorem \ref{onb}, we get that $\vert \hat{\phi}(\omega) \vert \leq 1$ for a.e. $\omega$. From the dominated convergence theorem, and Equation \ref{projection} we get that $\int_{U^*} F(\omega) \; d\omega=1$. As $0 \leq F(\omega) \leq 1$, $\lim_{j \rightarrow \infty}\vert \hat{\phi}(B^{-j}(\omega)) \vert=1$ a.e.
\end{proof}

\begin{theorem}
For a prime $p$, suppose that $\lbrace V_j: j \in \mathbb Z \rbrace$ is a Parseval FMRA. Then there exists an MRA $\lbrace \mathcal V_j: j \in \mathbb Z \rbrace$ such that $V_j \subseteq \mathcal V_j$, for each $j \in \mathbb Z$.
\end{theorem}

\begin{proof}
Assume that $\lbrace V_j: j \in \mathbb Z \rbrace$ is a Parseval FMRA with scaling function $\phi$. Note that $\hat{\phi}_{\Vert \omega}=0$, for $\omega \notin \eta(V_0)$. Let
$$
E_0=\lbrace \omega \in U^*:\hat{\phi}_{\Vert \omega}\neq 0 \rbrace,  E_j=\lbrace \omega \in U^*: \hat{\phi}_{\Vert B^{-j}\omega} \neq 0, \; \hat{\phi}_{\Vert B^{-m}\omega} = 0, \text{ for } 0 \leq m <j\rbrace.
$$
By Theorem \ref{scaling function fmra} (iii), $\lim_{j \rightarrow \infty}\vert \hat{\phi}(B^{-j}\omega) \vert=1$, for a.e. $\omega \in \tilde{G}^*$. Thus, $U^*=\cup_{j \geq 0}E_j$. Therefore, for each $\omega \in U^*$, there exists a unique $j(\omega) \in \mathbb N_0$ such that $\omega \in E_{j(\omega)}$. On $U^*$, let the function $\hat{\varphi}$ be defined by
$$
\hat{\varphi}(\omega \oplus B^{j(\omega)}h)=\hat{\phi}(B^{-j(\omega)}\omega \oplus h), \text{ for } h \in H^{\perp}, \text{ and } \hat{\varphi}(\omega \oplus h)=0, \text{ if } h \notin B^{j(\omega)}H^{\perp}.
$$
Then
$
\Vert \hat{\varphi}_{\Vert \omega} \Vert^2=1,
$
for a.e. $\omega \in U^*$. Therefore, $\varphi \in L^2(\tilde{G})$ and $\lbrace T_h \varphi : h \in H\rbrace$ is an orthonormal basis for $\mathcal V_0$, where $\mathcal V_j=\overline{\text{span}}\lbrace \mathbb D^jT_h\varphi : h \in H \rbrace$, for $j \in \mathbb Z$. For $\omega \in \eta(V_0)$, we have $\hat{\phi}_{\Vert \omega}=\hat{\varphi}_{\Vert \omega}$. Thus
\begin{equation}\label{phi and varphi}
\hat{\phi}_{\Vert \omega}=\mathbf{1}_{(\eta(V_0))^{\sim}}(\omega)\hat{\varphi}_{\Vert \omega}.
\end{equation}
Then we have $V_0 \subseteq \mathcal V_0$, since span$\lbrace \hat{\phi}_{\Vert \omega} \rbrace \subseteq $ span$\lbrace \hat{\varphi}_{\Vert \omega} \rbrace$ (from \eqref{phi and varphi}). Since $V_j=\mathbb D^j(V_0)$ and $\mathcal V_j=\mathbb D^j(\mathcal V_0)$, we have $V_j \subseteq \mathcal V_j$, for each $j \in \mathbb Z$. To show that $\lbrace \mathcal V_j:j \in \mathbb Z \rbrace$ is an MRA with a scaling function $\varphi$, we only need to check that $\varphi$ satisfies conditions (ii) and (iii) of Theorem \ref{scaling function mra}. From \eqref{phi and varphi}, we have
$$
\vert \hat{\phi}(\omega) \vert \leq \vert \hat{\varphi}(\omega) \vert, \text{ for a.e } \omega \in \tilde{G}^*.
$$
From (iii) of Theorem \ref{scaling function fmra}, we have
$$
1 \geq \vert \hat{\varphi}(B^{-j}\omega) \vert \geq \vert \hat{\phi}(B^{-j}\omega) \vert \rightarrow 1, \text{ as } j \rightarrow \infty, \text{ for a.e. } \omega \in \tilde{G}^*.
$$
Thus, $\lim_{j \rightarrow \infty}\vert \hat{\varphi}(B^{-j}\omega) \vert=1$, for a.e. $\omega \in \tilde{G}^*$. Now, we have to find $H^{\perp}$-periodic function $m_{\varphi} \in L^2(U^*)$ such that $\hat{\varphi}(B\omega)=m_{\varphi}(\omega)\hat{\varphi}(\omega)$, for a.e. $\omega \in \tilde{G}^*$, which is equivalent to $(\hat{\varphi}(B\omega \oplus Bh))_{h \in H^{\perp}}=m_{\varphi}(\omega)\hat{\varphi}_{\Vert\omega}$, for a.e. $\omega \in U^*$. Let $m$ be a low-pass filter for $\phi$ such that
$$
(\hat{\phi}(B \omega \oplus Bh))_{h \in H^{\perp}}=m(\omega)\hat{\phi}_{\Vert \omega}, \text{ for a.e. } \omega \in U^*.
$$
Let the $H^{\perp}$-periodic function $m_{\varphi}$ be defined by
\begin{equation*}
m_{\varphi}(\omega) =
\begin{cases}
m(\omega), & \text{if $\omega \in \eta(V_0) \cap B^{-1}((\eta(V_0))^{\sim})$}\\
1, & \parbox[t]{8.0cm}{if $\omega \in (\eta(V_0) \cap B^{-1}((U^* \setminus \eta(V_0))^{\sim}) \cap B^{-1}U^*)\cup (((U^*\setminus \eta(V_0)) \setminus B^{-1}((\eta(V_0))^{\sim})) \cap B^{-1}U^*)$}\\
0, & \text{otherwise}.
\end{cases}
\end{equation*}
Note that $U^*$ is the disjoint union of $\eta(V_0) \cap B^{-1}((\eta(V_0))^{\sim})$, $\eta(V_0) \cap B^{-1}((U^*\setminus \eta(V_0))^{\sim})$, $(U^* \setminus \eta(V_0)) \cap B^{-1}((\eta(V_0))^{\sim})$ and $(U^*\setminus\eta(V_0)) \cap B^{-1}((U^*\setminus \eta(V_0))^{\sim})$.

Suppose that $\omega \in \eta(V_0)$ and $B\omega \in \eta(V_0) \oplus H^{\perp}$. Now, as $\omega \in \eta(V_0)$, then $\hat{\phi}(\omega \oplus h)=\hat{\varphi}(\omega \oplus h)$, for every $h \in H^{\perp}$. Since $B\omega \in \eta(V_0)\oplus H^{\perp}$, therefore either $B\omega \in \eta(V_0)$, or $B\omega \ominus h' \in \eta(V_0)$, for some non-zero $h' \in H^{\perp}$. Let $B\omega \in \eta(V_0)$, then 
$$
\hat{\varphi}(B\omega\oplus Bh)=\hat{\phi}(B\omega \oplus Bh)=m_{\varphi}(\omega)\hat{\varphi}(\omega \oplus h),
$$
for every $h \in H^{\perp}$. Let $B\omega \notin \eta(V_0)$, then $B\omega \ominus h'\in \eta(V_0)$, for some non-zero $h' \in H^{\perp}$.
\begin{align*}
\hat{\varphi}(B\omega \oplus Bh)&=\hat{\phi}(B\omega \ominus h' \oplus Bh \oplus h')\\
&=m(\omega)\hat{\phi}(\omega \oplus h)\\
&=m_{\varphi}(\omega)\hat{\varphi}(\omega \oplus h), \text{ for every } h \in H^{\perp}.
\end{align*}

Next, suppose that $\omega \in \eta(V_0)$ and $B\omega \in (U^* \setminus \eta(V_0)) \oplus H^{\perp}$. Then $B\omega \notin \eta(V_0) \oplus H^{\perp}$ and $\hat{\phi}(\omega \oplus h)=\hat{\varphi}(\omega \oplus h)$, for every $h \in H^{\perp}$. If $B\omega \in U^*$, then $B\omega \in U^* \setminus \eta(V_0)$. From the definition of $\varphi$, for $h \in H^{\perp}$
$$
\hat{\varphi}(B \omega \oplus Bh)=\hat{\phi}(\omega \oplus h)=\hat{\varphi}(\omega \oplus h).
$$
Thus, for $\omega \in \eta(V_0) \cap B^{-1}((U^* \setminus \eta(V_0))^{\sim}) \cap B^{-1}U^*$, $\hat{\varphi}(B\omega \oplus Bh)=m_{\varphi}(\omega)\hat{\varphi}(\omega \oplus h)$. If $B\omega \notin U^*$, then $B\omega=\omega'\oplus \zeta.0$. Thus, $B\omega \ominus \zeta.0=\omega' \in U^*\setminus \eta(V_0)$. Hence, $\hat{\varphi}(B\omega \oplus Bh)=0$, since $Bh \oplus \zeta.0 \notin B^{j(B\omega \ominus \zeta.0)}H^{\perp}$. Thus, $\hat{\varphi}(B\omega \oplus Bh)=m_{\varphi}(\omega)\hat{\varphi}(\omega \oplus h)$.

Further, suppose that  $\omega \in U^*\setminus \eta(V_0)$ and $B\omega \in \eta(V_0) \oplus H^{\perp}$. Then $\hat{\varphi}(B\omega \oplus Bh)=\hat{\phi}(B\omega \oplus Bh)=m(\omega)\hat{\phi}(\omega \oplus h)=0$, for every $h \in H^{\perp}$ as $\omega \in U^* \setminus \eta(V_0)$. Next, suppose that $B\omega \ominus h' \in \eta(V_0)$, for some non-zero $h' \in H^{\perp}$. Then,
\begin{align*}
\hat{\varphi}(B\omega \oplus Bh)& = \hat{\phi}(B\omega \oplus Bh)\\
                                &=m(\omega)\hat{\phi}(\omega \oplus h) =0, \text{ for every } h \in H^{\perp}, \text{ as } \omega \in U^*\setminus \eta(V_0).
\end{align*}
Thus, $m_{\varphi}(\omega)\hat{\varphi}(\omega \oplus h)=0$, for $\omega \in (U^* \setminus \eta(V_0)) \cap B^{-1}((\eta(V_0))^{\sim})$.

For $\omega \in (U^* \setminus \eta(V_0)) \cap B^{-1}((U^* \setminus \eta(V_0))^{\sim}) \cap B^{-1}U^*$, similar to above we have $\hat{\varphi}(B\omega \oplus Bh)=m_{\varphi}(\omega)\hat{\varphi}(\omega \oplus h)$, for every $h \in H^{\perp}$ with $m_{\varphi}(\omega)=1$. Now, if $\omega \in (U^* \setminus \eta(V_0)) \cap B^{-1}((U^* \setminus \eta(V_0))^{\sim})$ and $B\omega \notin U^*$, then $B\omega = \omega' \oplus \zeta.0$. Thus, $B\omega \ominus \zeta.0=\omega' \in U^* \setminus \eta(V_0)$. Hence $\hat{\varphi}(B\omega \oplus Bh)=0$. Thus, for $\omega \in (U^* \setminus \eta(V_0)) \cap B^{-1}((U^* \setminus \eta(V_0))^{\sim})$ and $\omega \notin B^{-1}(U^*)$, with $m_{\varphi}(\omega)=0$ we get $\hat{\varphi}(B\omega \oplus Bh)=m_{\varphi}(\omega)\hat{\varphi}(\omega \oplus h)$.
\end{proof}

As another application of spectrum of shift invariant spaces, necessary and sufficient conditions are given for the existence of frame wavelet $\psi$ associated with an FMRA $\lbrace V_j:j \in \mathbb Z \rbrace$ with scaling function $\phi$ for the Cantor dyadic group $G$. In case of $G$, the dilation operator $\mathbb D: L^2(G) \rightarrow L^2(G)$ is defined as
$$
\mathbb D(f)(x)=2^{1/2}f(\rho x), \text{ for } f \in L^2(G), \; x \in G.
$$
The definition of FMRA for $L^2(G)$ is same as that for $L^2(\tilde{G})$. Notations and definitions related to shift invariant spaces for Cantor dyadic group are same as given in the beginning of Section \ref{6-2}. In case of $G$, translation by $n \in \Lambda$ is given as $T_nf(x)=f(x-n)$.

Let $\lbrace V_j:j \in \mathbb Z \rbrace$ be an FMRA in $L^2(G)$ with scaling function $\phi$. Let $\lbrace T_n\phi:n \in \Lambda \rbrace$ be the frame for $V_0=\overline{\text{span}}\{T_{n}\phi:n \in \Lambda\}$ with lower frame bound $\mathcal C$, and upper frame bound $\mathcal D$. %Note that, $\mathbb DT_{\rho n}f(x)=2^{1/2}f(\rho x -\rho n)=T_{n}\mathbb Df(x)$ for every $n\in \Lambda$.

Now, $V_{1}=\mathbb DV_{0}=\overline{\text{span}}\{\mathbb DT_{n}\phi:n \in \Lambda\}$. 
$$
T_{n}\mathbb D\phi(x)=2^{1/2}\phi(\rho x -\rho n) \text{ and }
T_{l}\mathbb DT_{1.0}\phi(x)=2^{1/2}\phi(\rho x -\rho l-1.0).
$$
Thus, $V_{0}=\langle \phi \rangle$ and $V_{1}=\langle \mathbb D\phi, \mathbb DT_{1.0}\phi \rangle=\overline{\text{span}}\{T_{n}\mathbb D\phi, T_{l}\mathbb DT_{1.0}\phi :n,l \in \Lambda\}$. Further,
%Since, $\phi \in V_{1}$, there exists $m_{0} \in L^{2}(D)$ such that $\hat{\phi}(\rho \omega)=m_{0}(\omega)\hat{\phi}(\omega)$. Here, $\{T_{n}\phi:n \in \Lambda\}$ is a frame for $V_{0}$ and
$\{T_{n}\mathbb D\phi, T_{l}\mathbb DT_{1.0}\phi :n,l \in \Lambda\}$, or $\{\mathbb DT_{n}\phi:n \in \Lambda\}$ is a frame for $V_{1}$ with the same frame bounds $\mathcal C$ and $\mathcal D$.

$\phi \in V_0 \subseteq V_{1}$ implies that $\phi(x)=\sum_{n \in \Lambda}d_{n}\phi(\rho x+n)$, or $\hat{\phi}(\rho \omega)=m(\omega)\hat{\phi}(\omega)$,  where $m \in L^2(D)$ is $\Lambda$-periodic. Note that the scalars $d_n$ are not unique, and hence so is $m$. Now, $\hat{V_{1}}_{\Vert\omega}=\text{span}\{\widehat{(\mathbb D\phi)}_{\Vert\omega},\widehat{(\mathbb DT_{1.0}\phi)}_{\Vert\omega}\}$, where
$$
\widehat{(\mathbb D\phi)}_{\Vert\omega}=(2^{-1/2}\hat{\phi}(\sigma(\omega+n)))_{n \in \Lambda},
$$
and
$$
\widehat{(\mathbb DT_{1.0}\phi)}_{||\omega}=(2^{-1/2}(-1)^{\omega_{0}+n_0}\hat{\phi}(\sigma(\omega+n)))_{n \in \Lambda}.
$$
Thus, $\hat{V_{1}}_{||\omega}=\text{span}\{(\hat{\phi}(\sigma(\omega+n)))_{n \in \Lambda}, ((-1)^{\omega_{0}}\hat{\phi}(\sigma(\omega+n)))_{n \in \Lambda}\}$, and \\
$\hat{V_{0}}_{||\omega}=\text{span}\{(\hat{\phi}(\omega+n))_{n \in \Lambda}\}$.

For $\omega \in D$, let $a^{\omega}=(a^{\omega}(n))_{n \in \Lambda}=(\hat{\phi}(\sigma(\omega+n)))_{n \in \Lambda}$.

Let
$$
a_{e}^{\omega}(n)=\left\{\begin{matrix}
a^{\omega}(n),&  \text{ if } n \in \Lambda_1\\
0, & \text{ if } n \in \Lambda \setminus \Lambda_1
\end{matrix}\right.,
$$
and
$$
a_{o}^{\omega}(n)=\left\{\begin{matrix}
 0,& \text{ if } n \in \Lambda_1\\
a^{\omega}(n), & \text{ if } n \in \Lambda \setminus \Lambda_1
\end{matrix}\right..
$$

Then, $a^{\omega}=a_{e}^{\omega}+a_{o}^{\omega}$ and $\langle a_{e}^{\omega}, a_{o}^{\omega}\rangle_{l^{2}(\mathbb{N}_{0})}=0.$
Let $b^{\omega}=(-1)^{\omega_{0}}a_{e}^{\omega}-(-1)^{\omega_{0}}a_{o}^{\omega}$. Therefore,
$$
\hat{V_{1}}_{||\omega}=\text{span}\{a^{\omega},b^{\omega}\}=\text{span}\{a_{e}^{\omega},a_{o}^{\omega}\}.
$$

Thus,
\begin{align*}
\eta(V_{1})&=\{\omega \in D: \text{dim} \hat{V_{1}}_{||\omega} \neq 0\}\\
&=\{\omega \in D: \sum_{n \in \Lambda_{1}}|\hat{\phi}(\sigma(\omega+n))|^{2} \neq 0; \text{ or } \sum_{n \in \Lambda \setminus \Lambda_{1}}|\hat{\phi}(\sigma(\omega+n))|^{2} \neq 0\}.
\end{align*}

Let $c^{\omega}=(c^{\omega}(n))_{n \in \Lambda}=(m(\sigma(\omega+n)) \hat{\phi}(\sigma(\omega+n)))_{n \in \Lambda}$, and $m_{e}^{\omega}=m(\sigma \omega)$, $m_{o}^{\omega}= m(\sigma \omega+0.1)$. Since $m$ is $\Lambda$-periodic, therefore
$$
m(\sigma\omega + \sigma n)=\left\{\begin{matrix}
m(\sigma\omega),& \text{ if } n \in \Lambda_1\\
m(\sigma\omega+0.1), & \text{ if } n \in \Lambda \setminus \Lambda_1
\end{matrix}\right..
$$

Now, $\hat{V_{0}}_{||\omega}=\text{span}\{c^{\omega}\}=\text{span}\{m_{e}^{\omega}a_{e}^{\omega}+m_{o}^{\omega}a_{o}^{\omega}\}$. Let
\begin{align*}
\Delta_{2}&=\{\omega \in D: \text{dim} \hat{V_{1}}_{||\omega}=2\}\\
&=\{\omega \in D: \sum_{n \in \Lambda_{1}}|\hat{\phi}(\sigma(\omega+n))|^{2} \neq 0 \text{ and } \sum_{n \in \Lambda \setminus \Lambda_{1}}|\hat{\phi}(\sigma(\omega+n))|^{2} \neq 0\},
\end{align*}
and
\begin{align*}
\Delta_{1}&=\{\omega \in D: \text{dim} \hat{V_{1}}_{||\omega} =1\} \\
&=\{\omega \in D: \text{ one of }\sum_{n \in \Lambda_{1}}|\hat{\phi}(\sigma(\omega+n))|^{2} \text{ and } \\
&\hspace{2.75cm}\sum_{n \in \Lambda \setminus \Lambda_{1}}|\hat{\phi}(\sigma(\omega+n))|^{2} \text{ is zero and the other is non-zero }\}.
\end{align*}

Now, $\eta(V_{0}) \subset \eta(V_{1})= \Delta_{2} \cup \Delta_{1}$, where $\Delta_{2}$ and $\Delta_{1}$ are disjoint sets. Let $\hat{V_1}_{\Vert \omega}=\hat{V_0}_{\Vert \omega} \oplus \hat{W_0}_{\Vert \omega}$. Suppose that $\omega \in \Delta_{2}$. If $m_{e}^{\omega}=m_{o}^{\omega}=0$, then $\hat{V_{0}}_{||\omega}=\{0\}$ and hence $\hat{W_{0}}_{||\omega}$ is two dimensional. Note that if $W_{0}$ is a principal shift invariant space, then $\hat{W_{0}}_{||\omega}$ must be one, or zero dimensional, depending on whether $\omega \in \eta(W_{0})$, or not. This gives us the following result.

\begin{lemma}
Let $E=\{\omega \in \Delta_{2}:m(\sigma\omega)=m(\sigma\omega+0.1)=0\}$. If $|E|>0$, then $W_{0}$ is not a principal shift invariant space. i.e. there does not exist $\psi \in W_{0}$ such that $W_{0}=\langle \psi \rangle$.
\end{lemma}

Now, suppose that $|E|=0$. We have to construct $\psi \in W_{0}$ such that $W_{0}=\langle \psi \rangle$ and $\{T_{n}\psi:n \in \Lambda\}$ is a frame for $W_{0}$. For this, we will construct $(\hat{\psi}(\omega+n))_{n \in \Lambda} \in \hat{W_{0}}_{||\omega}$, for $\omega \in D=(D \setminus \eta(V_{1}))\cup \Delta_{2} \cup \Delta_{1}$.

Let $\omega \in D\setminus \eta(V_{1})$, then $\hat{V_{1}}_{||\omega}=\lbrace 0\rbrace$ and hence $\hat{W_{0}}_{||\omega}=\lbrace 0\rbrace$. Let $(\hat{\psi}(\omega+n))_{n \in \Lambda}=0$, for $\omega \in D \setminus \eta(V_{1})$. Now, suppose that $\omega \in \Delta_{2}$, then $\text{dim}\hat{V_{0}}_{||\omega}=1$. Thus, $\text{dim}\hat{W_{0}}_{||\omega}=1$.

Let
$$
(\hat{\psi}(\omega+n))_{n \in \Lambda}=(-1)^{\omega_{0}}\overline{m_{o}^{\omega}}||a_{o}^{\omega}||^{2}a_{e}^{\omega}(n)-(-1)^{\omega_{0}}\overline{m_{e}^{\omega}}||a_{e}^{\omega}||^{2}a_{o}^{\omega}(n) \in \hat{W_{0}}_{||\omega},
$$
then
\begin{align*}
\hat{\psi}(\omega+n)&=(-1)^{\omega_{0}}\overline{m(\sigma\omega+0.1)} \sum_{n' \in \Lambda\setminus \Lambda_1}|\hat{\phi}(\sigma(\omega+n'))|^{2}a_{e}^{\omega}(n)\\
&-(-1)^{\omega_{0}}\overline{m(\sigma\omega)} \sum_{n' \in \Lambda_1}|\hat{\phi}(\sigma(\omega+n'))|^{2}a_{o}^{\omega}(n).
\end{align*}

If $n \in \Lambda_{1}$, then
$$
\hat{\psi}(\omega+n)=(-1)^{\omega_{0}}\overline{m(\sigma\omega+0.1)} \sum_{n' \in \Lambda \setminus \Lambda_{1}}|\hat{\phi}(\sigma(\omega+n'))|^{2}\hat{\phi}(\sigma(\omega+n)),
$$
and if $n \in \Lambda \setminus \Lambda_{1}$, then
$$
\hat{\psi}(\omega+n)=-(-1)^{\omega_{0}}\overline{m(\sigma\omega)} \sum_{n' \in \Lambda_{1}}|\hat{\phi}(\sigma(\omega+n'))|^{2}\hat{\phi}(\sigma(\omega+n)),
$$
or, for each $n \in \Lambda$ we can write
$$
\hat{\psi}(\omega+n)=(-1)^{\omega_{0}+n_0}\overline{m(\sigma(\omega+n)+0.1)} \sum_{n' \in \Lambda \setminus \Lambda_{1}}|\hat{\phi}(\sigma(\omega+n'+n))|^{2}\hat{\phi}(\sigma(\omega+n)).
$$
Thus, for $\omega \in \Delta_{2}+\Lambda$,
$$
\hat{\psi}(\omega)=(-1)^{\omega_{0}}\overline{m(\sigma\omega+0.1)} \sum_{n' \in \Lambda \setminus \Lambda_{1}}|\hat{\phi}(\sigma(\omega+n'))|^{2}\hat{\phi}(\sigma(\omega)).
$$

Let $\omega \in \Delta_{1}=(\Delta_{1}\cap \eta(V_{0}))\cup (\Delta_{1}\cap \eta(V_{0})^{c})$.

If $\omega \in \Delta_{1} \cap \eta(V_{0}),$ then $\text{dim}\hat{V_{0}}_{||\omega}=1$ and hence $\text{dim}\hat{W_{0}}_{||\omega}=0$. Let $\hat{\psi}(\omega)=0$, for every $\omega \in (\Delta_{1}\cap \eta(V_{0}))+\Lambda$. Next, if $\omega \in \Delta_{1} \cap (\eta(V_{0}))^c$, then $\text{dim}\hat{V_{0}}_{||\omega}=0$ and $\hat{W_{0}}_{||\omega}=\hat{V_{1}}_{||\omega}$. In this case, we know that one of $a_{e}^{\omega}$ and $a_{o}^{\omega}$ is zero and the other is non-zero. Let $\hat{\psi}(\omega)=\hat{\phi}(\sigma(\omega))$, for $\omega \in (\Delta_{1} \cap \eta(V_{0})^{c})+\Lambda$.

Thus, $\hat{\psi}(\omega)=m_{\psi}(\omega)\hat{\phi}(\sigma(\omega))$, where

$$
m_{\psi}(\omega)=\left\{\begin{matrix}
(-1)^{\omega_{0}}\overline{m(\sigma\omega+0.1)} \sum_{n' \in \Lambda \setminus\Lambda_{1}}|\hat{\phi}(\sigma(\omega+n'))|^{2}, & \text{if }\omega \in \Delta_{2}+\Lambda \\
1, &\text{ if } \omega \in (\Delta_{1} \cap \eta(V_{0})^{c})+\Lambda \\
0, & \text{ otherwise }
\end{matrix}\right..
$$

Then, $m_{\psi}\in L^{2}(D)$ by result similar to Theorem \ref{frame seq char1} for Cantor dyadic group. We will show that $m_{\psi} \in L^{\infty}(D)$, so that $\psi \in L^{2}(G)$.

\begin{lemma}\label{1}
If $|E|=0$ and $\omega \in \Delta_{2}$ then $\sigma(\omega)$ and $\sigma(\omega)+0.1$ are in $\eta(V_{0})$. Also, for a.e. $\omega \in \Delta_2$, we have $\omega \in \eta(V_0)$.
\end{lemma}

\begin{proof}
Since $\omega \in \Delta_{2}$, $\sum_{n \in \Lambda_{1}}|\hat{\phi}(\sigma(\omega+n))|^{2} \neq 0$ and $\sum_{n \in \Lambda \setminus \Lambda_{1}}|\hat{\phi}(\sigma(\omega+n))|^{2}\neq 0$. Thus, $\sum_{n \in \Lambda}|\hat{\phi}(\sigma\omega+n)|^{2} \neq 0$ and $\sum_{n \in \Lambda}|\hat{\phi}(\sigma\omega+n+0.1)|^{2}\neq 0$. Hence, we have $\sigma\omega \in \eta(V_{0})$, $\sigma\omega+0.1 \in \eta(V_{0})$. Further,

\begin{equation}\label{11}
\sum_{n \in \Lambda}|\hat{\phi}(\omega+n)|^{2}=|m(\sigma\omega)|^{2}\sum_{n \in \Lambda}|\hat{\phi}(\sigma(\omega)+n)|^{2}+|m(\sigma\omega+0.1)|^{2}\sum_{n \in \Lambda}|\hat{\phi}(\sigma\omega+n+0.1)|^{2}
\end{equation}

As $|E|=0$, therefore except on a set of measure zero, $\sum_{n \in \Lambda}|\hat{\phi}(\omega+n)|^{2} \neq 0$ and hence a.e. $\omega$ is in $\eta(V_{0})$.
\end{proof}

Let $\omega \in \Delta_{2}$ and $|E|=0$. Then by Lemma \ref{1} and Theorem \ref{frame seq char1}, we have
$$
\mathcal C \leq \sum_{n \in \Lambda}|\hat{\phi}(\sigma\omega+n)|^{2} \leq \mathcal D, \quad \mathcal C \leq \sum_{n \in \Lambda}|\hat{\phi}(\sigma\omega+0.1+n)|^{2} \leq \mathcal D. 
$$
Further, for a.e. $\omega \in \Delta_2$
$$
\mathcal C \leq \sum_{n \in \Lambda}|\hat{\phi}(\omega+n)|^{2} \leq \mathcal D.
$$

Now, from equation \eqref{11} $\mathcal C[|m(\sigma\omega)|^{2}+|m(\sigma\omega+0.1)|^{2}] \leq \mathcal D$, or $|m(\sigma\omega)|^{2}+|m(\sigma\omega+0.1)|^{2} \leq \frac{\mathcal D}{\mathcal C}$ a.e.. Thus, $m_{\psi}$ is essentially bounded function and hence $\psi \in L^{2}(\tilde{G})$. Further, we have
$$
\hat{W_{0}}_{||\omega}=\text{span}\{\hat{\psi}_{||\omega}\}, \quad \text{ for a.e. } \omega \in D.
$$

\begin{lemma}
If $|E|=0$, then $W_{0}$ is a principal shift invariant space.
\end{lemma}

Next, we have to prove that $\{T_{n}\psi:n \in \Lambda\}$ is a frame for $W_{0}$. From the definition of $\hat{\psi}$, $\eta(W_{0})=\Delta_{2} \cup (\Delta_{1} \cap (\eta(V_{0})^{c}))$ (modulo null sets).

For a.e. $\omega \in \Delta_{2}$,
\begin{align*}
\sum_{n \in \Lambda}|\hat{\psi}(\omega+n)|^{2}&=|m(\sigma(\omega)+0.1)|^{2}|\sum_{n' \in \Lambda \setminus \Lambda_{1}}|\hat{\phi}(\sigma(\omega+n'))|^{2}|^{2}\sum_{n \in \Lambda_{1}}|\hat{\phi}(\sigma(\omega+n))|^{2}\\
&+|m(\sigma(\omega))|^{2}|\sum_{n' \in \Lambda_{1}}|\hat{\phi}(\sigma(\omega+n'))|^{2}|^{2}\sum_{n \in \Lambda \setminus \Lambda_{1}}|\hat{\phi}(\sigma(\omega+n))|^{2}\\
&=(\sum_{n \in \Lambda}|\hat{\phi}(\omega+n)|^{2})(\sum_{n \in \Lambda}|\hat{\phi}(\sigma\omega+n)|^{2})(\sum_{n \in \Lambda}|\hat{\phi}(\sigma\omega+0.1+n)|^{2}).
\end{align*}

Hence,
$$
\mathcal C^{3} \leq \sum_{n \in \Lambda}|\hat{\psi}(\omega+n)|^{2} \leq \mathcal D^{3}.
$$

Now, suppose that $\omega \in \Delta_{1} \cap(\eta(V_{0}))^{c}$, then $\hat{\psi}(\omega)=\hat{\phi}(\sigma(\omega))$. Since $\omega \in \Delta_{1}$, either $a_{e}^{\omega}$, or $a_{o}^{\omega}$ is non-zero, i.e. either $\sigma\omega$, or $\sigma\omega+0.1$ belongs to $\eta(V_{0})$.

$$
\sum_{n \in \Lambda}|\hat{\psi}(\omega+n)|^{2}=\sum_{n \in \Lambda}|\hat{\phi}(\sigma(\omega+n))|^{2}=\sum_{n \in \Lambda}|\hat{\phi}(\sigma\omega+n))|^{2}+\sum_{n \in \Lambda}|\hat{\phi}(\sigma\omega+n+0.1))|^{2}.
$$

Therefore, we have $\mathcal C \leq \sum_{n \in \Lambda}|\hat{\psi}(\omega+n)|^{2} \leq \mathcal D$. Let $\mathcal C'=\min\lbrace\mathcal C, \mathcal C^3\rbrace$ and $\mathcal D'=\max\lbrace \mathcal D, \mathcal D^3 \rbrace$. Then, $\{T_{n}\psi:n \in \Lambda\}$ is a frame for $W_{0}$ with bounds $\mathcal C'$ and $\mathcal D'$.

\begin{lemma}
Suppose that $E=E_{m}$ and $|E|=0$. Let $m' \in L^{2}(D)$ be such that $\hat{\phi}(\rho \omega)=m'(\omega)\hat{\phi}(\omega)$. Then, $|E_{m'}|=0$, where
$E_{m'}=\{\omega \in \Delta_{2}:m'(\sigma(\omega))=m'(\sigma(\omega)+0.1)=0\}$.
\end{lemma}

\begin{proof}
Follows from Lemma \ref{1}.
\end{proof}

\begin{corollary}
If $|E|=0$, then for a.e. $\omega \in \Delta_{2}$,
$$
\frac{\mathcal C}{\mathcal D} \leq |m(\sigma(\omega))|^{2}+|m(\sigma(\omega)+0.1)|^{2} \leq \frac{\mathcal D}{\mathcal C}.
$$
\end{corollary}

\noindent Proof of the corollary follows from Lemma \ref{1} and Equation \eqref{11}. Converse of this result, that is, \textit{if the above two inequalities hold for a.e. $\omega \in \Delta_2$, then $|E|=0$}, is trivially true.

Thus, the necessary and sufficient conditions for the existence of frame wavelet can be given as;

\begin{theorem}
Let $(V_{j})$ be an FMRA with scaling function $\phi$. Let $m \in L^{2}(D)$ be a $\Lambda$-periodic function such that $\hat{\phi}(\rho\omega)=m(\omega)\hat{\phi}(\omega)$. Then there exists a frame wavelet $\psi \in W_{0}$, where $V_{1}=V_{0}\bigoplus W_{0}$ (orthogonal direct sum), that is $\{T_{n}\psi:n \in \Lambda\}$ is a frame for $W_{0}$ and hence $\{2^{j/2}\psi(\rho^{j}.-n):n \in \Lambda, j \in \mathbb{Z}\}$ is a frame for $L^{2}(G)$ if and only if one of the following equivalent conditions is satisfied.

\begin{enumerate}
\item[(i)] $m(\sigma\omega)$ and $m(\sigma\omega+0.1)$ are not simultaneously zero for a.e. $\omega \in \Delta_{2}$.
\item[(ii)] There exists positive constants $\mathcal C$ and $\mathcal D$ such that for a.e. $\omega \in \Delta_{2}$.
$$
\mathcal C \leq |m(\sigma(\omega))|^{2}+|m(\sigma(\omega)+0.1)|^{2} \leq \mathcal D.
$$
\end{enumerate}
\end{theorem}

\end{document}